\documentclass[times]{oupau}

\newtheorem*{theorem'}{Theorem 1.7}
\def\co{\colon\thinspace}
\newtheorem{remark}[theorem]{Remark}

\begin{document}

\title{A spectral sequence of the Floer cohomology of symplectomorphisms of trivial polarization class}
\shorttitle{A spectral sequence for the Floer cohomology of certain symplectomorphisms}

\author{Kristen Hendricks\affil{1}}
\abbrevauthor{K. Hendricks}
\headabbrevauthor{Hendricks, K.}

\address{%
\affilnum{1}Department of Mathematics, UCLA, 520 Portola Plaza\\ 
Los Angeles, CA 90095}

\correspdetails{hendricks@math.ucla.edu}

\received{}
\revised{}
\accepted{}

\communicated{}

\begin{abstract} Let $M$ be an exact symplectic manifold equal to a symplectization near infinity and having stably trivializable tangent bundle, and $\phi\co M\rightarrow M$ be an exact symplectomorphism which, near infinity, is equal to either the identity or the symplectization of a contactomorphism $\hat{\phi}$ such that neither $\hat{\phi}$ nor $\hat{\phi}^2$ has fixed points. We give conditions under which Seidel and Smith's localization theorem for Lagrangian Floer cohomology implies the existence of a spectral sequence from $\mathit{HF}(\phi^2)\otimes \mathbb Z_2((\theta))$ to $\mathit{HF}(\phi)\otimes \mathbb Z_2((\theta))$.\end{abstract}

\maketitle

\section{Introduction}

Fixed point Floer cohomology is an invariant introduced by Floer \cite{MR987770} and refined by Dostoglou and Salamon \cite{MR1297130} which associates to a symplectic manifold $(M,\omega)$ and symplectomorphism $\phi$ (with one of several possible choices of technical conditions) a graded $\mathbb Z_2$-vector space $\mathit{HF}(\phi)$. Floer introduced it hoping to study the Arnol'd conjecture, which he proved in the positively monotone case \cite{MR987770}. The theory is an invariant of the Hamiltonian isotopy class of $\phi$, and has been used to study the symplectic mapping class group by Seidel \cite{MR2441414}, Khovanov and Seidel \cite{MR1862802}, Keating  \cite{MR3217627}, and others. Many computations of of the Floer cohomology of particular symplectomorphisms in the two-dimensional case have been made by Cotton-Clay \cite{MR2529943, MR2713919}, Eftekhary \cite{MR2131640}, and Gautschi \cite{MR2039162}.

The definition of $\mathit{HF}(\phi)$ involves counting pseudoholomorphic cylinders connecting the (perturbed, nondegenerate) fixed points of $\phi$, or equivalently doing Morse theory on a twisted free loop space of $M$. However, there is an identification of $\mathit{HF}(\phi)$ with the Lagrangian Floer cohomology of the graph $\Gamma_{\phi} = \{(a, \phi(a)): a \in M\}$ of $\phi$ and the diagonal $\Delta = \{(a,a): a \in M\}$ in the manifold $M\times M^{-}$ when this theory is well-defined. Here $M^- $ is the symplectic manifold $(M, -\omega)$. (For more on this identification, see Remark \ref{remark:equivalence}.) The purpose of this paper is use this identification and Seidel and Smith's localization theorem for Lagrangian Floer cohomology to give hypotheses under which there exists a spectral sequence from $\mathit{HF}(\phi^2) \otimes \mathbb Z_2((\theta))$ to $\mathit{HF}(\phi) \otimes \mathbb Z_2((\theta))$, where $\mathbb Z_2((\theta))$ is the ring $\mathbb Z_2[[\theta]](\theta^{-1})$.

Our assumptions will be as follows. Let $(M, J, \omega)$ be an exact symplectic manifold with symplectic form $\omega = d\lambda$ and compatible almost complex structure $J$. We ask that $M$ admits a strictly plurisubharmonic function $f\co M \rightarrow \mathbb R$ with compact critical set, so that near infinity $M$ is the symplectization of the contact manifold $M_c = f^{-1}(c)$ for sufficiently large $c$. We do not allow this structure to vary (this is equivalent to starting with a Liouville domain with contact-type boundary, and passing to its symplectization). Let $\phi:M\rightarrow M$ be an exact symplectomorphism. We will require that near infinity $\phi$ is either equal to the identity, or equal to the symplectization of a contactomorphism $\hat{\phi}: M_c \rightarrow M_c$ such that $\hat{\phi}$ and $\hat{\phi}^2$ have no fixed points. We will see that it follows from work of Khovanov and Seidel for the Lagrangian Floer cohomology of Lagrangians with Legendrian boundary in Liouville domains that the Lagrangian Floer cohomology $\mathit{HF}(\Delta, \Gamma_{\phi})$ is well-defined for either of these conditions.  (If $\phi$ is the identity near infinity, we must first deform $\Gamma_{\phi}$ to have compact intersection with $\Delta$; see Section 2 for details.)

Our first major hypothesis is that the tangent bundle $TM$ is stably trivialized as a symplectic vector bundle. We then consider the map $\Phi \co M\rightarrow Sp(\infty)$ into the symplectic group induced by the action of $\phi_*$ on the stabilized tangent bundle. After picking a deformation retract from $Sp(\infty)$ to $U$, we see that $[\Phi]$ is a class in $K^1(M)$. We call it the polarization class. (This is a slight abuse of notation: the class $[\Phi]$ depends on the trivialization of $TM$, although whether it vanishes does not.)

\begin{theorem} \label{thm:main} If $TM$ is stably trivialized as a symplectic vector bundle and the polarization class $[\Phi]$ is trivial, there is a spectral sequence with $E^1$ page $\mathit{HF}(\phi^2) \otimes \mathbb Z_2((\theta))$ and $E^{\infty}$ page $\mathbb Z_2((\theta))$-isomorphic to $\mathit{HF}(\phi)\otimes \mathbb Z_2((\theta))$.
\end{theorem}

The assumption that $TM$ is stably trivialized and the map $\Phi\co M\rightarrow Sp(\infty)$ is nulhomotopic has appeared previously in the literature as a possible restriction on Floer cohomology constructions; it is used in work of Cohen, Jones, and Segal as a prerequisite to constructing a homotopy theory for $\mathit{HF}(\phi)$ \cite{MR1362832}.

Theorem \ref{thm:main} has the following corollaries. Let $M$ and $\phi$ satisfy the hypotheses above.

\begin{corollary} \label{corollary:inequality} If the hypotheses of Theorem \ref{thm:main} are satisfied, there is a rank inequality

\vskip 3mm
\hfill $\displaystyle \operatorname{rk}\left(\mathit{HF}(\phi^2)\right) \geq \operatorname{rk}(\mathit{HF}(\phi)).$ 
\end{corollary}

Now let $\operatorname{Symp}^c(M)$ be the group of symplectomorphisms of $M$ which are equal to the identity outside a compact set. The following is a direct consequence of the rank inequality of Corollary \ref{corollary:inequality}. Note that if $[\phi]$ is trivial in $\pi_0(\operatorname{Symp}^c(M))$, then $\operatorname{rk}(\mathit{HF}(\phi)) = \operatorname{rk}(H^*(M; \mathbb Z_2))$ \cite{MR1432464, MR2424172}.

\begin{corollary} Let $M$ and $\phi$ be a symplectic manifold and exact symplectomorphism satisfying the hypotheses of Theorem \ref{thm:main} and such that $\phi$ is equal to the identity outside a compact set. Suppose that \begin{align*}
\operatorname{rk}(\mathit{HF}(\phi))> \operatorname{rk}(H^*(M; \mathbb Z_2)). \end{align*}
Then the classes $[\phi^{2^n}] \in \pi_0(\operatorname{Symp}^c(M))$ are all nontrivial.
\end{corollary}

There are also two special cases in which the hypotheses of Theorem \ref{thm:main} are automatically satisfied.

\begin{corollary} Let $M = T^*S^n$ for $n$ even, and $\phi:M\rightarrow M$ any symplectomorphism which, near infinity, is equal to the identity or the symplectization of a contactomorphism $\hat{\phi}$ such that $\hat{\phi}$ and $\hat{\phi}^2$ have no fixed points. Then there is a spectral sequence whose $E^1$ page is $\mathit{HF}(\phi^2) \otimes \mathbb Z_2((\theta))$ and whose $E^{\infty}$ page is $\mathbb Z_2((\theta))$-isomorphic to $\mathit{HF}(\phi) \otimes \mathbb Z_2((\theta))$.
\end{corollary}

\begin{proof} Since $H^1(M) = 0$, $\phi$ is automatically exact. Furthermore, $T(T^*S^n)$ is stably trivializable since $TS^n$ is. The set of maps $[T^*S^n,Sp(\infty)]$ is equal to $\langle S^n, U \rangle = \pi_n(U)=0$ since $n$ is even. Hence the map $\Phi$ must be nulhomotopic.
\end{proof}

The same holds if we have a \emph{plumbing} of even-dimensional spheres. The notion of symplectic plumbing was introduced by Gompf\cite{MR1270434} after being suggested by Gromov \cite{MR864505}; a very concrete definition can be found in \cite[Definition 2.1]{MR2984118}.

\begin{corollary} Let $M$ be a symplectic plumbing of cotangent bundles of even-dimensional spheres along some tree, $\phi:M\rightarrow M$ any symplectomorphism which, near infinity, is equal to the identity or the symplectization of a contactomorphism $\hat{\phi}$ such that $\hat{\phi}$ and $\hat{\phi}^2$ have no fixed points. Then there is a spectral sequence whose $E^1$ page is $\mathit{HF}(\phi^2) \otimes \mathbb Z_2((\theta))$ and whose $E^{\infty}$ page is $\mathbb Z_2((\theta))$-isomorphic to $\mathit{HF}(\phi) \otimes \mathbb Z_2((\theta))$.
\end{corollary}

The proof of Theorem \ref{thm:main} relies first on the following proposition, which is a consequence of a well-known strategy for unfolding holomorphic curves. A proof can be found in Section 2. Let $M^4 = M^-\times M\times M^- \times M$, and let $\Delta\Delta = \{(a,b,b,a)\}$ and $\Gamma_{\phi}\Gamma_{\phi} = \{(a,\phi(a),b,\phi(b)): a,b \in M\}$, which are each exact Lagrangians in $M^4$. We have the following.

\begin{proposition} \label{propn:isomorphism} There is an isomorphism
\vskip 3mm
\hfill $\displaystyle\mathit{HF}(\Gamma_{\phi}\Gamma_{\phi}, \Delta\Delta) \simeq \mathit{HF}(\Gamma_{\phi^2}, \Delta).$
\end{proposition}

Therefore, we have an identification of $\mathit{HF}(\phi^2)$ with $\mathit{HF}(\Gamma_{\phi}\Gamma_{\phi}, \Delta\Delta)$, and we will concentrate our attention on the latter version of this theory. 

The second, more involved, ingredient of our proof is a theorem of Seidel and Smith \cite{MR2739000} which gives a localization spectral sequence for Lagrangian Floer cohomology under certain rigid technical conditions. Let $N$ be an exact symplectic manifold which is convex at infinity, and let $L_0$ and $L_1$ be two exact Lagrangian submanifolds such that $L_0 \cap L_1$ is compact and all holomorphic disks counted by the Lagrangian Floer differential for $L_0$ and $L_1$ lie in a compact set. Let $\tau \co N \rightarrow N$ be a symplectic involution which preserves $L_0$ and $L_1$ setwise. Then the fixed set of $\tau$ is a symplectic manifold $N^{\operatorname{inv}}$ containing Lagrangians $L_0^{\operatorname{inv}}$ and $L_1^{\operatorname{inv}}$ which are the fixed sets of $L_0$ and $L_1$ respectively. There is an additional, highly nontrivial, hypothesis on $(N,L_0,L_1)$ called the existence of a \textit{stable normal trivialization}. We will say more about this hypothesis in Section 4.

\begin{theorem}\label{thm:localization}\cite[Theorem 20]{MR2739000} Suppose $(N,L_0,L_1)$ has a stable normal trivialization. Then there is a spectral sequence whose first page is $\mathit{HF}(L_0,L_1)\otimes \mathbb Z_2[[\theta]]$ and a map $\Lambda: E^{\infty} \rightarrow \mathit{HF}(L_0^{\operatorname{inv}},L_1^{\operatorname{inv}})\otimes \mathbb Z_2[[\theta]]$ from the $E^{\infty}$ page which becomes an isomorphism after tensoring with $\theta^{-1}$.
\end{theorem}

The $E^{\infty}$ page of the spectral sequence is the Borel or equivariant cohomology of $(N,L_0,L_1)$ together with the involution $\tau$. In our particular case, we consider the manifold $N = M^4=M\times M^{-}\times M\times M^{-}$ with the symplectic involution $\tau(a,b,c,d)=(c,d,a,b)$. Our Lagrangians and invariant sets are as follows.
\begin{align*}
M^4 &= \{(a,b,c,d): a,b,c,d \in M\} &&& (M^4)^{\mathrm{inv}}&= \iota(M \times M^{-})=\{(a,b,a,b): a,b \in M\}\\
\Gamma_{\phi}\Gamma_{\phi} &= \{(a,\phi(a),b,\phi(b): a,b \in M\}&&&  \Gamma_{\phi}\Gamma_{\phi}^{\mathrm{inv}} &= \iota(\Gamma_{\phi}) = \{(a,\phi(a), a, \phi(a)): a\in M\}\\
\Delta\Delta &= \{(a,b,b,a): a,b \in M\} &&& \Delta\Delta^{\mathrm{inv}} &= \iota(\Delta) = \{(a,a,a,a): a \in M\}
\end{align*}

Here $\iota$ is the natural diagonal symplectic embedding $(M \times M^-, 2\omega \oplus \omega^{-})\hookrightarrow M^4$. In Section 2 we will check that the symplectic conditions of Theorem \ref{thm:localization} are met, and in Section 3 we will check that $(M^4,\Gamma_{\phi}\Gamma_{\phi},\Delta\Delta)$ has a stable normal trivialization. Once we have satisfied ourselves that all these hypotheses are met, we see that Proposition \ref{propn:isomorphism} and Theorem \ref{thm:localization} together imply the spectral sequence 
\begin{align*}
\mathit{HF}(\phi^2) \otimes \mathbb Z_2((\theta)) \simeq \mathit{HF}(\Gamma_{\phi}\Gamma_{\phi}, \Delta\Delta)\otimes \mathbb Z_2((\theta)) \rightrightarrows \mathit{HF}(\Gamma_{\phi},\Delta)\otimes \mathbb Z_2((\theta)) \simeq \mathit{HF}(\phi) \otimes \mathbb Z_2((\theta))
\end{align*}
\noindent described in Theorem \ref{thm:main}.

This paper is organized as follows. In Section 2 we briefly discuss the definition of the Lagrangian Floer cohomology of a symplectomorphism, paying particular attention to issues arising from noncompact Lagrangians, check that the symplectic conditions of Theorem \ref{thm:localization} are met, and give a proof of Proposition \ref{propn:isomorphism}. In Section 3 we introduce Seidel and Smith's localization theorem in more detail and show that our setup satisfies their triviality conditions. We conclude with a remark about the broader context of Proposition \ref{propn:isomorphism} and possible future directions for research.

\section{Exact conical Lagrangians and $\mathit{HF}(\phi)$} \label{Section:Floer Cohomology}

Floer cohomology is an invariant for Lagrangian submanifolds in a symplectic manifold introduced by Floer \cite{MR965228, MR933228, MR948771}. Many versions of the theory exist; in this section, we review some of the setup of Lagrangian Floer cohomology for manifolds which are isomorphic to the symplectization of a contact manifold near infinity, and discuss the reasons that $\mathit{HF}(\phi)$ is well-defined for the maps $\phi$ we consider in Theorem \ref{thm:main}. (We do not, however, define the Floer differential.)

\subsection{Floer cohomology for conical Lagrangians}

We will work exclusively with Lagrangians that are \textit{conical} in the following sense. Let $(N, \omega, \lambda, J)$ be an exact symplectic manifold with $\omega = d\lambda$ the symplectic form and $J$ an $\omega$-compatible almost complex structure on $N$. Suppose that $N$ is convex at infinity and of finite type; that is, suppose that $N$ admits an exhausting function $f\co N \rightarrow [0,\infty)$ with $\lambda = -d^{\mathbb C} f$ and $\omega = -dd^{\mathbb C}f$, and there exists some $C>0$ such that all critical points of $f$ occur in $N_{<(C-1)}=f^{-1}([0,C-1))$. Then $N_{C} = f^{-1}(C)$ is a contact manifold with contact form $\alpha_C = \lambda|_{N_C}$ and $N_{\geq C}$ is symplectomorphic to the symplectization of $N_C$. Let $Z_{f}$ be the Liouville vector field and $\kappa\co[0,\infty)\times N_C \rightarrow N$ its flow. 

\begin{definition}
Let $L \subset N$ be a Lagrangian submanifold such that $L\cap N_C$ is a Legendrian submanifold of $N_C$. Then $L$ is said to be \emph{conical} if $\kappa([0, \infty)\times (L\cap N_C)) = L_{\geq C}$, where $L_{\geq C} = L \cap N_{\geq C}$.
\end{definition}

Let $L_0,L_1$ be two conical Lagrangians in $N$ with compact intersection. (For convenience, we increase $C$ until $L_0 \cap L_1 \subseteq N_{<(C-1)}$). The first step in defining the Lagrangian Floer cohomology of $L_0$ and $L_1$ is to show that $L_0$ and $L_1$ can be deformed via Lagrangian isotopy supported on $N_{\leq C}$ to have transverse intersection. For this result, and several following, we will use Khovanov and Seidel's exposition of the Floer cohomology of exact Lagrangians with Legendrian boundary in a symplectic manifold with contact type boundary \cite{MR1862802}. We now introduce their setup.

Suppose that $(V,\partial V)$ is an exact symplectic manifold with contact type boundary. Let $\omega = d\lambda$ be the symplectic form on $V$ and $\alpha = \lambda|_{\partial V}$ be the contact form. Let $A_0,A_1$ be two exact Lagrangian submanifolds such that $\Lambda_i = A_i \cap \partial V$ is Legendrian for $i=0,1$ and $\Lambda_0 \cap \Lambda_1 = \emptyset$. Let $Z$ be the outward-pointing Liouville vector field on $\partial V$, and $\kappa \co (-r,0] \times \partial V \rightarrow V$ be the negative time flow of $Z$. Let $R$ be the Reeb vector field on $\partial V$.

\begin{definition}
We say that $A_i$ is \emph{$\kappa$-compatible} if there is some $\epsilon>0$ such that $\kappa^{-1}(A_i) \cap ([-\epsilon,0] \times \partial V) = [-\epsilon,0]\times \Lambda_i$.
\end{definition}

Of course, the case we are interested in is $(V, \partial V) = (N_{\leq C}, N_C)$, and $A_i = L_i \cap N_{\leq C}$ for $i=0,1$. We have the following isotopy lemmas.

\begin{lemma} \label{Lemma:Isotopy} \cite[Lemma 5.2]{MR1862802}(a) Any Lagrangian submanifold $A_i$ of $(V,\partial V)$ with Legendrian boundary can be deformed, rel $\partial V$, into a $\kappa$-compatible Lagrangian.\\
(b) Let $(A_t)_{0\leq t \leq 1}$ be a Lagrangian isotopy such that $A_0$ and $A_1$ are $\kappa$-compatible and $A_t\cap \partial V$ is Legendrian for all $t$. Then there is an isotopy $(A_t')_{0\leq t \leq 1}$ of $\kappa$-compatible Lagrangians with Legendrian boundary with the same endpoints such that $A_t \cap \partial V = A_{t'} \cap \partial V$ for all $t$. If $(A_t)$ is exact then $(A_{t'})$ may also be chosen to be exact.
\end{lemma}

\begin{lemma}\cite[Lemma 5.6]{MR1862802} Let $A_0$ and $A_1$ be exact Lagrangian submanifolds of $(V, \partial V)$ such that $\partial A_0 \cap \partial A_1 = \emptyset$. There are Lagrangian submanifolds $A_0'$ and $A_1'$ such that for $i=0,1$,  $A_i$ is joined to $A_i'$ by an exact isotopy rel $\partial V$, each $A_i'$ is $\kappa$-compatible, and $A_0' \pitchfork A_1'$.\end{lemma}

This gives us the desired transverse intersection result, as follows:  If $L_0$, $L_1$ are two exact Lagrangian submanifolds of $N$ with intersection contained in $N_{<C}$ which are conical outside of $N_{C}$, then their restrictions to the symplectic manifold with boundary $(N_{\leq C}, \partial N_{\leq C}) = (V, \partial V)$ are two $\kappa$-compatible Lagrangians with Legendrian boundary. (This relies on the fact that all critical points of $f$ occur on $N_{< (C-1)}$, and not arbitrarily close to $N_C$.) This implies that $(L_0\cap V) = A_0$ and $(L_1 \cap V) = A_1$ can be deformed by a $\kappa$-compatible exact Lagrangian isotopy to intersect transversely, without changing $L_i\cap(\partial V) = (L_i)_C$ for $i=0,1$. By $\kappa$-compatibility, these isotopies can be regarded as exact Lagrangian isotopies of $L_0$ and $L_1$ to $L_0'$ and $L_1'$ which preserve $(L_i)_{\geq C}$. Therefore $L_0'\pitchfork L_1'$ is contained in $N_{<C}$.

Once this deformation has been accomplished, we can say a few words about the definition of the Floer cohomology of $L_0$ and $L_1$. The Floer cochain complex is $\mathit{CF}(L_0,L_1) = \mathbb Z_2\langle L_0' \cap L_1' \rangle$. Let ${\bf J}$ be a family of complex structures which perturbs $J$ and respects the structure of $N$ as a symplectization near infinity. The Floer differential counts finite energy strips $u:  \mathbb R \times [0,1]$ which are solutions to Floer's equations
\begin{align*}
\frac{\partial u}{\partial s} + J_t(u) \frac{\partial u}{\partial t} = 0 & & u(\mathbb R \times \{0\}) \subset L_0' & & u(\mathbb R \times \{1\}) \subset L_1' & &\lim_{s \rightarrow \infty} u(s,t) = x & & \lim_{s\rightarrow -\infty} u(s,t) =y
\end{align*}
\noindent with respect to ${\bf J}$, up to reparametrization. (Here $x,y \in L_0' \cap L_1'$.) We let $\mathcal M({\bf J})$ be the space of such curves for some ${\bf J}$ which is regular, that is, such that the operator $D_{\bf J}u$ arising which gives a linearization of Floer's equations is surjective for all $u \in \mathcal({\bf J})$. Khovanov and Seidel show that for $\kappa$-compatible Lagrangians $A_0', A_1'$ in $(V, \partial V)$ with finite intersection, there is a maximum modulus principle \cite[Lemma 5.5]{MR1862802} implying that all $u \in \mathcal M({\bf J})$ lie in a compact set contained in $V \backslash \partial V$. After this is established, standard arguments show the Floer differential, and thus $\mathit{HF}(A_0,A_1)$, is well-defined. The same follows for $\mathit{HF}(L_0,L_1)$. (Through a slight abuse of notation, we will sometimes say $\mathit{HF}(L_0,L_1) = \mathit{HF}(A_0,A_1)$, since the two theories have the same generators and pseudoholomorphic curves.) Furthermore, one can show this is invariant of the choices involved. In particular, we have the following.

\begin{lemma} \cite[Proposition 5.10]{MR1862802} $\mathit{HF}(L_0,L_1)$ is invariant under exact Lagrangian isotopy supported on $N_{< C}$.
\end{lemma}

In fact for the next section we will require a slightly stronger invariance result that allows for isotopies that move $\partial A_0$.

\begin{lemma} \label{lemma:s-invariance} \cite[Proposition 5.11]{MR1862802}  Let $(A_0)_{0 \leq s \leq 1}$ be an isotopy of exact Lagrangians with Legendrian boundary, and $A_1$ an exact Lagrangian with Legendrian boundary such that $\partial A_{0,s} \cap \partial A_{1} = \emptyset$ for all $s$. Then $\mathit{HF}(A_{0,s},A_1)$ is independent of $s$ up to isomorphism.\end{lemma}


\subsection{The diagonal, the graph, and other important Lagrangians}

In this section, we show the Lagrangians we are interested in are indeed exact and conical, and discuss how to deform them to have compact intersection.

As at the beginning of this section, let $(M, \omega, \lambda, J)$ be an exact symplectic manifold with $\omega = d\lambda$ the symplectic form and $J$ an $\omega$-compatible almost complex structure on $M$. Suppose that $M$ is convex at infinity and of finite type; that is, suppose that $M$ admits an exhausting function $f\co M \rightarrow [0,\infty)$ with $\lambda = -d^{\mathbb C} f$ and $\omega = -dd^{\mathbb C}f$, and there exists some $c>0$ such that all critical points of $f$ occur in $M_{<(c-1)}=f^{-1}([0,c-1))$. Then as previously, $M_{c}$ is a contact manifold with contact form $\alpha_c = \lambda|_{M_c}$ and $M_{\geq c}$ is symplectomorphic to the symplectization of $M_c$.

Let $\phi \co M \rightarrow M$ be an exact symplectomorphism which, outside a compact set $K$, is either the identity or equal to the symplectization of a contactomorphism $\hat{\phi}$ such that $x \neq \hat{\phi}(x)\neq \hat{\phi}^2(x)$.(For simplicity, we assume $K \subset M_{<(c-1)}$, although once we have made this adjustment we no longer allow $c$ to change). We consider the manifold $N = M \times M^{-}$ with the symplectic form $\omega \oplus (-\omega) = d(\lambda \oplus (-\lambda))$ and complex structure $J \oplus -J$. This manifold admits an exhausting function $f \oplus f$. Notice that $\omega \oplus (-\omega) = -dd^{\mathbb C}(f \oplus f)$, because the complex structure has a different sign on each factor. Moreover, the critical points of $f\oplus f$ lie in $(f \oplus f)^{-1}([0,C-2)) = N_{<C-2}$, where $C=2c$.

Consider the Lagrangians $\Delta = \{(a,a): a \in M\}$ and $\Gamma_{\phi} = \{(a, \phi(a)): a \in M\}$. We see that the restriction of $\lambda \oplus (-\lambda)$ to $\Delta$ is identically zero, so $\Delta$ is exact. Moreover, since $\phi$ is an exact symplectomorphism, $\phi_*\lambda = \lambda + \eta$, with $\eta$ an exact one-form, so if $(v, \phi_*(v)) \in T\Gamma_{\phi}$, then $(\lambda\oplus -\lambda)(v, \phi_*v) = \lambda(v) - \phi^*\lambda(v) = -\eta(v)$, so $\lambda \oplus (-\lambda)$ restricted to $\Gamma_{\phi}$ is exact.

We can easily see that the submanifold $\Delta \subset M\times M^{-}$ is conical. Indeed, since the exhausting function on $M\times M^{-}$ is $f \oplus f$, if $Z_f$ is the Liouville vector field on $M$, then $Z_{f \oplus f} = Z_f\oplus Z_f$ is the Liouville vector field on $M \times M^{-}$. In particular, the Liouville flow on $M\times M^{-}$ is the product of the Liouville flow on each copy of $M$ individually, hence preserves the diagonal. Moreover, since $\lambda \oplus (-\lambda)|_{\Delta} = 0$, the intersection of $\Delta$ with any $M_C$ is Legendrian.

Now we must discuss $\Gamma_{\phi}$. We have two cases. First, suppose $\phi$ is equal to the symplectization of a contactomorphism with no fixed points outside of a compact set. Then for all $d > c$, $M_c$ is identified with $M_d$ under the Liouville flow, and in particular if the time $t$ Liouville flow takes $x \in M_c$ to $y \in M_d$, then by our assumptions that $\phi$ is the symplectization of a contactomorphism, it must take $\phi(x) \in M_C$ to $\phi(y) \in M_d$. Since the Liouville flow on $M\times M^{-}$ is the product of the Liouville flows on each copy of $M$, this implies that $\Gamma_{\phi}$ is preserved by the Liouville flow on $M \times M^{-}$. Furthermore, for any $D>C$, let $D=2d$. Then if $(x, \phi(x)) \in (M \times M^{-})_{D}$, since $f(x) = f(\phi(x))$, we must have $f(x) = f(\phi(x)) = d$. In particular, $(x, \phi(x))\in M_d \times M_d$, implying that if $(v,\phi_*(v))\in T\Gamma_{\phi} \cap T(M\times M^{-})_{D}$, then $(\lambda \oplus -\lambda)|_{(M\times M^{-})|_{D}}(v, \phi_*(v)) = (\alpha_d \oplus -\alpha_d)(v, \phi_*(v)) = \alpha_d(v) - \alpha_d(\phi_*v) = 0$, where the last step follows because $\phi$ preserves the contact form. Ergo $\Gamma_{\phi} \cap (M\times M^{-})_{D}$ is Legendrian.

Moreover, the intersection $\Delta \cap \Gamma_{\phi}$ is contained in $M_{\leq C}$, hence is compact. However, if $\phi$ is equal to the identity outside $M_{\leq c}$, we have $\Gamma_{\phi} = \Delta$ on $M_{\geq c}$. We will need to perturb $\Gamma_{\phi}$ by a Hamiltonian isotopy to make the intersection compact.

We begin by setting up some notation on the manifold $M$. For any $e>c$, notice that if one identifies $M_e$ with $M_c$ via the Liouville flow, the contact form $\alpha_e = \lambda|_{M_e}$ is identified with $\frac{e}{c}\alpha_c$. Let $X_f$ be the Hamiltonian vector field of $f \co M \rightarrow [0, \infty)$. Then on $M_c$, we have $X_f = cR_{\alpha_c}$, and on $M_e$, we have $X_f = eR_{\alpha_e}$. Choose an $s>0$ sufficiently small that $2sc$ is less than the period of all Reeb orbits on the contact manifold $M_c$. (This is always possible since the set of periods of Reeb orbits on a contact manifold attains a positive minimum, cf. \cite[page 109]{MR2797558}.) It follows that $se$ is not the period of any Reed orbit on $M_e$. Therefore the time $s$ flow of $X_f$ on $M_{\geq c}$ has no fixed points.

Now let us construct a suitable perturbation of $\Gamma_{\phi}$. Consider a smooth $h \co \mathbb R \rightarrow \mathbb R$ which is nondecreasing, equal to zero on $(-\infty,c+ \frac{1}{2})$, and equal to $s$ for $e \geq c+\frac{3}{4}$. Then consider a Hamiltonian $H \co M \rightarrow \mathbb R$ which is given by $h(f(x))f(x)$. Let $\psi_1$ be the time one flow of $H$, so that when $e>c+\frac{3}{4}$, on $M_e$ the map $\phi$ is the time $se$ flow of $R_{\alpha_e}$ (or equivalently, the time $s$ flow of $X_f$) and has no fixed points. Since $\psi_1$ and $\phi$ have disjoint support, $\psi_1$ commutes with $\phi$. We replace $\Gamma_{\phi}$ with $\Gamma_{\psi_1 \circ \phi} = \Gamma_{\phi \circ \psi_1}$. This is exact and conical by the same arguments as for $\Gamma_{\phi}$, and we see that $\Delta \cap \Gamma_{\psi_1 \circ \phi} = \emptyset$ on $M_{\geq c+1}$.

\begin{definition}
Let $\phi \co M\rightarrow M$ be an exact symplectomorphism such that $\phi$ is equal to the identity outside of a compact set. We say that the Floer cohomology $\mathit{HF}(\phi)$ of the symplectomorphism $\phi \co M \rightarrow M$ is the Lagrangian Floer cohomology $\mathit{HF}(\Gamma_{\psi_1 \circ \phi}, \Delta)$, for a map $\psi_1$ as chosen above.
\end{definition}

Lemma \ref{lemma:s-invariance} implies that this definition is independent of our choice of $s$ and subsequently of $h$, as long as $s$ is sufficiently small, since any two choices give Lagrangians $\Gamma_{\psi_1 \circ \phi}$ which are related by exact Lagrangian isotopy.

However, observe that our definition does depend on the structure of $M$ as a symplectization. This is equivalent to studying symplectomorphisms on with Liouville domains with contact-type boundary (and passing to their symplectizations where appropriate). However, the formulation given here seems more natural from the point of view of taking products, and is also better-adapted to applying Seidel and Smith's theory.

So far we have talked about the definition of $\mathit{HF}(\phi)$. Let's take a moment to lay some groundwork for the other Lagrangian Floer computation we will be interested in. Consider the manifold $M^4 = M \times M^{-} \times M \times M^{-}$ with plurisubharmonic function $f^{4}\co M^4\rightarrow \mathbb R$ and consequent symplectic form and primitive one-form. Let $C' = 4c$. Then $(M^4)_{\geq C'}$ is a symplectization of the contact manifold $M^4_{C'}$. Consider the Lagrangians $\Delta\Delta = \{(a,b,b,a): a,b \in M\}$ and $\Gamma_{\phi}\Gamma_{\phi} = \{(a,\phi(a),b,\phi(b)): a, b \in M \}$. The Lagrangians $\Delta\Delta$ and $\Gamma_{\phi}\Gamma_{\phi}$ are both products of conical Lagrangians in product symplectic manifolds (in the case of $\Gamma_{\phi}\Gamma_{\phi}$, via grouping the first and fourth factors and the second and third factors), hence conical. Let us consider their intersection. There are two cases.

First, suppose that on $M_{\geq c}$, the map $\phi$ is equal to the symplectization of a contactomorphism $\hat{\phi}$ such that $a \neq \hat{\phi}(a) \neq \hat{\phi}^2(a)$. Then suppose $(a,b,b,a) = (a,\phi(a), b, \phi(b))$ is a point in $\Gamma\Gamma \cap \Delta_{\phi}\Delta{\phi}$ which lies in $(M^4)_{\geq C'}$. Then $2f(a)+2f(b)>C'$, implying that either $f(a)>c$ or $f(b)>c$. Without loss of generality, let $f(a)>c$. But the equality between the two points implies that $\phi(a)=b$ and $\phi(b)=a$, so $\phi^2(a)=a$. This is impossible, because $\phi^2$ has no fixed points on $M_{\geq c}$. We conclude that $\Delta\Delta \cap \Gamma_{\phi}\Gamma_{\phi}$ is contained in $M^4_{\leq C'}$.

Now consider the case the $\phi$ is the identity on $M_{\geq c}$. Unlike $\Delta$ and $\Gamma_{\phi}$ in $M\times M^{-}$, the Lagrangians $\Delta\Delta$ and $\Gamma_{\phi}\Gamma_{\phi}$ are not identical outside of a compact set. However, we claim a very similar deformation $\Gamma_{\phi}\Gamma_{\phi}$ can be used to ensure that the intersection of the two Lagrangian lies in a compact subset of $M^4$, as follows. If $H \co M\rightarrow \mathbb R$ is the Hamiltonian on $M$ defined previously, let $\psi_{\frac{1}{2}}$ be the time $\frac{1}{2}$ flow of $\widetilde{H}$. Then we replace $\Gamma_{\phi}\Gamma_{\phi}$ with $\Gamma_{\psi_{\frac{1}{2}} \circ \phi}\Gamma_{\psi_{\frac{1}{2}}\circ \phi}$.


We claim that $\Delta\Delta$ and $\Gamma_{\psi_{\frac{1}{2}}\circ \phi}\Gamma_{\psi_{\frac{1}{2}}\circ \phi}$ do not intersect outside of $M_{< C'+4}$. For suppose that there is some $(a,b,b,a) \in (\Delta\Delta \cap \Gamma_{\psi_{\frac{1}{2}} \circ \phi},\Gamma_{\psi_{\frac{1}{2}} \circ \phi}) \cap M^4_{\geq C'+4}$. This implies that both $a$ and $b$ are fixed points of $(\psi_{\frac{1}{2}}\circ \phi)^2 = \psi_1 \circ \phi^2$. However, we know that $2(f(a)+f(b))>C'+4$, so at least one of $f(a)$ and $f(b)$ is greater than $c+1$. Without loss of generality, say it is $a$. Then $\phi(a)=a$, so $\psi_1(a) = a$. But this is impossible, since $\psi_1$ has no fixed points on $M_{>c+1}$. Therefore $\Gamma_{\psi_{\frac{1}{2}} \circ \phi,\frac{s}{2}}\Gamma_{\psi_{\frac{1}{2}}\circ \phi} \cap \Delta\Delta$ is contained in $M_{\leq C'+4}$, hence is compact. By the same arguments as previously, $\mathit{HF}(\Gamma_{\psi_{\frac{1}{2}} \circ \phi}\Gamma_{\psi_{\frac{1}{2}}\circ \phi}, \Delta\Delta)$ is well-defined.

\begin{remark} \label{rmk:cases} At this point we pause for a remark about our perturbations of $\Gamma_{\phi}$. Invariance under $s$ is extremely important to our construction for the following reason: under the involution $\tau(a,b,c,d)=(c,d,a,b)$, the fixed set of $\Gamma_{\psi_{\frac{1}{2}} \circ \phi} \Gamma_{\psi_{\frac{1}{2}} \circ \phi}$ is $\Gamma_{\psi_{\frac{1}{2}}\circ \phi}$, whereas we will see in Proposition \ref{propn:isomorphism} that $\mathit{HF}(\Gamma_{\psi_{\frac{1}{2}} \circ \phi} \Gamma_{\psi_{\frac{1}{2}} \circ \phi}, \Delta\Delta)$ is identified with $\mathit{HF}(\Gamma_{\psi_1 \circ \phi^2}, \Delta)$. Ergo the spectral sequence of Theorem \ref{thm:main} goes from $\mathit{HF}(\Gamma_{\psi_1 \circ \phi^2}, \Delta)\otimes \mathbb Z_2((\theta))$ to $\mathit{HF}(\Gamma_{\psi_{\frac{1}{2}}\circ \phi}, \Delta)$. Fortunately, both theories are independent of the choice of sufficiently small $s$.
\end{remark}

\subsection{Floer cohomology with the diagonal and the proof of Proposition \ref{propn:isomorphism}} Finally, we turn our attention to the proof of Proposition \ref{propn:isomorphism}. We first present a lemma whose proof is very similar to the identification between $\mathit{HF}(\phi)$ and $\mathit{HF}(\Gamma_{\phi}, \Lambda)$; the formulation we quote here is from Ganatra \cite[Proposition 8.2]{MR3121862}.

\begin{lemma} \label{lemma:product}
Let $L_0,L_1$ be exact conical Lagrangian subspaces of $(N,\omega)$ with intersection contained in a compact set $K \subset N$. Let $\Delta \subset N\times N^{-}$ be the diagonal subspace. There is an isomorphism
\vskip 3mm
\hfill $\displaystyle \mathit{HF}(L_0,L_1) \simeq \mathit{HF}(L_0\times L_1, \Delta).$
\end{lemma}

\noindent Let us sketch the proof of this lemma. First, observe that since the Liouville flow on $N\times N^-$ is split, $L_0\times L_1$ is conical if and only if $L_0$ and $L_1$ are. Next, $(L_0\times L_1)\cap \Delta = \{(a,a): a \in L_0\cap L_1\}$, hence is compact. If we assume that we have already perturbed $L_0$ and $L_1$ along an exact Lagrangian isotopy to intersect transversely, the intersection $(L_0\times L_1) \cap \Delta$ is also transverse. Choose a perturbation ${\bf J} = J_t$ of the complex structure on $N$ such that $D_{\bf J}u$ is regular for every pseudoholomorphic $u \in \mathcal M({\bf J})$ and ${\bf J}$ is compatible with the structure of the symplectization near infinity. Let $\tilde{\bf J} = J_{\frac{t}{2}} \oplus -J_{\frac{1-t}{2}}$. Let $v:\mathbb R\times [0,1] \rightarrow N\times N^{-}$. Then we may \textit{unfold} the holomorphic strip $v$ into two coordinate pseudoholomorphic strips $(u_1(\frac{s}{2},\frac{t}{2}), u_2(\frac{s}{2},\frac{1-t}{2}))$ such that each $u_1(s,0)\in L_0$, $u_2(s,1) \in L_1$, and $u_1(s,1) = u_2(s,0)$ for all $t$. We can glue together $u_1$ and $u_2$ to obtain a map $u \co \mathbb R \times [0,1] \rightarrow N$; this map is $C^1$ since both $u_1$ and $u_2$ solve Floer's equation, and therefore by elliptic regularity, $u$ must in fact be smooth. Moreover, $u$ is regular if and only if $v$ is regular. Conversely,  given a pseudoholomorphic strip $u \co \mathbb R \times [0,1] \rightarrow N$, one may \emph{fold} the strip, defining $v \co \mathbb R \times [0,1] \rightarrow N \times N^{-}$ via $v(s,t) = \left(u(2s, 2t) , u(2s, 2t-1)\right)$, where smoothness of $u$ implies smoothness of $v$ and once again regularity of $v$ is equivalent to regularity of $u$. This relationship gives a bijection between pseudoholomorphic strips counted by the differential on $\mathit{CF}(L_0, L_1)$ and pseudoholomorphic strips counted by the differential on $\mathit{CF}(L_0\times L_1, \Delta)$.

We are now ready to prove Proposition \ref{propn:isomorphism}.

\begin{proof}[Proof of Proposition \ref{propn:isomorphism}]

Our goal is to show that $\mathit{HF}(\Gamma_{\phi}\Gamma_{\phi}, \Delta\Delta) \simeq \mathit{HF}(\Gamma_{\phi}^2, \Delta)$. First, we apply $\phi$ to the second and third factors of $M\times M^-\times M \times M^-$. This map preserves $\Delta\Delta$, since it takes any point $(a,b,b,a)\in \Delta\Delta$ to $(a,\phi(b),\phi(b),a)\in \Delta\Delta$.  The image of any point $(a,\phi(a),b,\phi(b)) \in \Gamma_{\phi}\Gamma_{\phi}$ under this symplectomorphism is $(a,\phi^2(a),\phi(b),\phi(b))$, so this symplectomorphism has the effect of replacing $\Gamma_{\phi}\Gamma_{\phi}$ with $\Gamma_{\phi^2}\Delta = \{(a,\phi^2(a),b,b): a,b \in M\}$. (Notice that if $\phi$ was the identity near infinity and we perturbed $\phi$ as above, we have replaced $\Gamma_{\psi_{\frac{1}{2}}\circ \phi}\Gamma_{\psi_{\frac{1}{2}}\circ \phi}$ with $\Gamma_{\psi_1 \circ \phi^2}\Delta$.) This symplectomorphism $\operatorname{id}\oplus \phi \oplus \phi \oplus \operatorname{id}$ carries pseudoholomorphic curves with respect to ${\bf J}$ to pseudoholomorphic curves with respect to the complex structure $(\operatorname{id}\oplus \phi^{-1}\oplus \phi^{-1}\oplus \operatorname{id})^*{\bf J} = {\bf J}'$. Indeed, ${\bf J}'$ is regular if ${\bf J}$ was. The only issue is that this push-forward complex structure may not respect the structure of $M^4$ at infinity, however, it is still the case that all $u \in \mathcal({\bf J})$ are contained in a compact subset of $M^4$, so if necessary we perturb ${\bf J}'$ outside of this subset to make it compatible with the structure of $M^4$ as a symplectization near infinity. This does not change the pseudoholomorphic curves.  Ergo there is an isomorphism $\mathit{HF}(\Gamma_{\phi}\Gamma_{\phi},\Delta\Delta) \simeq \mathit{HF}(\Gamma_{\phi^2}\Delta, \Delta\Delta)$.

To finish the proof, we appeal to Lemma \ref{lemma:product}. Let $N=M\times M^-$, $L_0=\Gamma_{\phi^2}$, and $L_1 = \Delta$. Then $N\times N^- = M \times M^-\times M^- \times M$, with diagonal ${\bf \Delta} = (a,b,a,b)$. After rearranging (but not changing the sign of) the third and fourth factors, we see that $\mathit{HF}(\Gamma_{\phi^2}\Delta,\Delta\Delta) \simeq \mathit{HF}(L_0\times L_1, {\bf \Delta}) \simeq \mathit{HF}(L_0,L_1) = \mathit{HF}(\Gamma_{\phi^2},\Delta)$.
\end{proof}

\begin{remark} \label{remark:equivalence} While $\mathit{HF}(\phi)$ has been introduced in terms of Lagrangian Floer cohomology in this paper, this is, as mentioned in the introduction, not the usual definition. Floer's first results for fixed points of symplectomorphisms come from Lagrangian Floer cohomology, cf. \cite{MR965228}, but he soon wrote down a more technically flexible definition in \cite{MR987770}, which is now considered standard. This was subsequently used in \cite{MR1297130}. We say a few words about that approach here. Let $\phi \co M \rightarrow M$ be a symplectomorphism of a compact symplectic manifold. If necessary, apply a small Hamiltonian perturbation to $\phi$ such that it has isolated and nondegenerate fixed points. Then the chain complex $C_*$ for the Floer cohomology of $\phi$ is generated by the fixed points of $\phi$ over some field. (In certain cases this field may be taken to be $\mathbb Z_2$, but more generally it must be a Novikov field.) The Floer differential $\partial_{\phi}$ counts maps $v \co \mathbb R \times \mathbb R \rightarrow M$ satisfying
\begin{align*}
v(s, t+1) = f(v(s,t)) \quad \quad \frac{\partial v}{\partial_s} + J_t(v)\frac{\partial v}{\partial_t} = 0 \quad \quad \lim_{s \rightarrow \infty} v(s,t) = x \quad \quad \lim_{s \rightarrow -\infty} v(s,t) = y
\end{align*}
where $x$ and $y$ are fixed points of $\phi$ and ${\bf J}'$ is family of $\omega$-compatible almost complex structures ${\bf J}'$ achieving transverality with the property that $J_{t+1} = f_* \circ J_t \circ (f^{-1})_*$. (Here the prime is intended only as a reminder that for this family of complex structures we allow $t \in \mathbb R$, rather than just $[0,1]$.)  As usual, these pseudoholomorphic cylinders are counted up to the action of $s$. When it is the case that the Lagrangian Floer cohomology of $\Gamma_{\phi}$ and $\Delta$ is well-defined, the relation between the two definitions is quite similar to the argument of Lemma \ref{lemma:product}. It is convenient to look at $\mathit{HF}(\Delta, \Gamma_{\phi})$ in $M^{-} \times M$, which is canonically isomorphic to $\mathit{HF}\Gamma_{\phi}, \Delta)$ in $M \times M^-$. There is a clear bijection between the generators of the chain complexes: a fixed point of $\phi$ corresponds to a point in $\Delta \cap \Gamma_{\phi}$. Given a family of almost complex structures ${\bf J}'$ on $M$, we consider the family of almost complex structures $-J_{\frac{1-t}{2}} \oplus J_{\frac{1+t}{2}}$ on $M^- \times M$, for $t \in [0,1]$. Then if $v \co \mathbb R \rightarrow M$ is a pseudoholomorphic cylinder counted by $\partial_{\phi}$, we define a pseudoholomorphic strip $u \co \mathbb R \times [0,1] \rightarrow M^- \times M$ via $u(s,t) = (v(\frac{s}{2}, \frac{1-t}{2}), v(\frac{s}{2}, \frac{1+t}{2}))$. The map $u$ is regular if and only if $v$ is. Conversely, given a pseudoholomorphic $u \co \mathbb R \times [0,1] \rightarrow M^- \times M$, we can unfold its components to get a cylinder $v(s,t) = u(2s, 1-2t)$ for $t \in [0, \frac{1}{2}]$ and $v(s,t) = u(2s, 2t-1)$ for $t \in [\frac{1}{2},1]$, and define $v(s,t)$ via periodicity for all other $t$. As in Lemma \ref{lemma:product}, by elliptic regularity, this map is smooth, and again, regularity of $v$ is equivalent to regularity of $u$.\end{remark}

\section{Existence of a stable normal trivialization}

In this section, we discuss the concept of a stable normal trivialization, the major technical hypothesis of Theorem \ref{thm:localization}. We then show that the manifold $M^4 = M \times M^- \times M \times M^-$ with the Lagrangians $\Gamma_{\phi}\Gamma_{\phi}$ and $\Delta\Delta$ and involution $\tau(a,b,c,d)=(c,d,a,b)$ carries a stable normal trivialization.

Recall that $N$ is an exact symplectic manifold which is convex at infinity, and $L_0$ and $L_1$ are two Lagrangian submanifolds which are exact with the prpoerty that $L_0 \cap L_1$ is compact, and all holomorphic curves counted by the Floer differential lie in a compact set. As in the introduction, let $\tau \co N \rightarrow N$ be a symplectic involution which preserves $L_0$ and $L_1$ setwise. Then the fixed set of $\tau$ is a symplectic manifold $N^{\operatorname{inv}}$ containing Lagrangians $L_0^{\operatorname{inv}}$ and $L_1^{\operatorname{inv}}$, the fixed sets of $L_0$ and $L_1$ respectively.

We need to set up a little notation to introduce Seidel and Smith's technical conditions. Let $N(N^{\operatorname{inv}})$ be the normal bundle of $N^{\operatorname{inv}} \subset N$, and let $\Upsilon(N^{\operatorname{inv}})$ be the pullback of $N(N^{\operatorname{inv}})$ to $N^{\operatorname{inv}}\times I$, where $I$ is the unit interval $[0,1]$. For $i=0,1$, let $NL_i^{\operatorname{inv}}$ be the Lagrangian normal bundle of $L_i^{\operatorname{inv}}\subset L_i$. Let $NL_i^{\operatorname{inv}}\times \{j\}$ denote the copy of $NL_i^{\operatorname{inv}}$ which is a subbundle of $\Upsilon(N^{\operatorname{inv}})|_{L_i^{\operatorname{inv}} \times \{j\}}$.

\begin{definition} \cite[Definition 18]{MR2739000} A \emph{stable normal trivialization} of $(N, L_0, L_1)$ consists of the following data:
\begin{itemize}

\item A unitary trivialization $\Psi \co \Upsilon(N^{\operatorname{inv}}) \oplus \epsilon_{\mathbb C}^{m} \rightarrow \mathbb C^{n+m}=\mathbb C^k$.

\item A Lagrangian subbundle $\Lambda_0 = \Upsilon(N^{\operatorname{inv}})|_{L_0\times I}$ such that $\Lambda_0|_{L_0 \times \{0\}} = (NL_0^{\operatorname{inv}} \times \{0\}) \oplus \epsilon_{\mathbb R}^{m}$ and $\Psi(\Lambda_0|_{L_0 \times \{1\}}) = \mathbb R^k$. 

\item A Lagrangian subbundle $\Lambda_1 = \Upsilon(N^{\operatorname{inv}})|_{L_1\times I}$ such that $\Lambda_1|_{L_1 \times \{0\}} = (NL_1^{\operatorname{inv}} \times \{0\}) \oplus i\epsilon_{\mathbb R}^{m}$ and $\Psi(\Lambda_0|_{L_1 \times \{1\}}) = i\mathbb R^k$. \mbox{\qedhere}

\end{itemize}
\end{definition}

We quote Seidel and Smith's main result (which also appeared in the introduction) again for the reader's convenience.

\begin{theorem'}\cite[Theorem 20]{MR2739000} Suppose that $(N, L_0, L_1)$ carries a stable normal trivialization. Then there is a spectral sequence whose $E^2$ page is $\mathit{HF}(L_0,L_1)\otimes \mathbb Z_2((\theta))$ and whose $E^{\infty}$-page is $\mathbb Z_2((\theta))$-isomorphic to $\mathit{HF}(L_0^{\operatorname{inv}}, L_1^{\operatorname{inv}})$. In particular, there is a rank inequality
\vskip 3mm
\hfill $\displaystyle \operatorname{rk}(\mathit{HF}(L_0,L_1)) \geq \operatorname{rk}(\mathit{HF}(L_0^{\operatorname{inv}}, L_1^{\operatorname{inv}})).$
\end{theorem'}

Before showing that $(M^4, \Gamma_{\phi}\Gamma_{\phi}, \Delta\Delta)$ has a stable normal trivialization, let us pause for a quick note on the proof of Theorem \ref{thm:localization} in the case of noncompact Lagrangians. In the proof, the stable normal trivialization is used to deform the Lagrangians $L_0$ and $L_1$ in such a way that pseudoholomorphic strips inside $N^{\operatorname{inv}}$ remain regular when considered as psedudoholomorphic strips in $N$. Importantly, this deformation fixes the invariant sets $L_i^{\operatorname{inv}}$ for $i=0,1$. When the fixed sets are noncompact, we would like to have this deformation be compactly supported. Therefore, we choose a compact set $K \subset N$ that contains the image of all $u \subset \mathcal M({\bf J})$ (which always exists because the intersection of $L_0^{\operatorname{inv}}$ and $L_1^{\operatorname{inv}}$ is compact and $N$ is convex at infinity) and interpolate between the full deformation given by the stable normal trivialization on a neighborhood of $K$ and the identity near infinity.

We now show that, under the hypotheses of Theorem \ref{thm:main}, $(M^4, \Gamma_{\phi}\Gamma_{\phi},\Delta\Delta)$ carries a stable normal trivialization. Recall that let $M$ is a symplectic manifold which is exact, convex at infinity, and whose tangent bundle is stably trivializable, and $\phi \co M\rightarrow M$ is an exact symplectomorphism which, near infinity, is the symplectization of a contactomorphism $\hat{\phi}$ such that neither $\hat{\phi}$ nor $\hat{\phi}^2$ has fixed points. As per Remark \ref{rmk:cases}, up to perturbation this is the only case we need to check. Recall from the introduction that the manifolds salient to our investigation have the following form.
\begin{align*}
N&=M^4 = \{(a,b,c,d): a,b,c,d \in M\} &&& N^{\mathrm{inv}}&=(M^4)^{\mathrm{inv}}= \iota(M \times M^{-})=\{(a,b,a,b): a,b \in M\}\\
L_0&=\Gamma_{\phi}\Gamma_{\phi} = \{(a,\phi(a),b,\phi(b): a,b \in M\}&&& L_0^{\mathrm{inv}}&=\Gamma_{\phi}\Gamma_{\phi}^{\mathrm{inv}} = \iota(\Gamma_{\phi}) = \{(a,\phi(a), a, \phi(a): a\in M\}\\
L_1&=\Delta\Delta = \{(a,b,b,a): a,b \in M\} &&& L_1^{\mathrm{inv}}&=\Delta\Delta^{\mathrm{inv}} = \iota(\Delta) = \{(a,a,a,a): a \in M\}
\end{align*}
Here $M^4$ indicates $M \times M^{-} \times M \times M^{-}$, and $\iota$ is the diagonal embedding. Ergo we have tangent bundles as follows.
\begin{align*}
TM^4 &= \{(v,w,x,y): v,w,x,y \in TM\}&&&T(M \times M^{-}) &= \{(v,w,v,w): v,w \in TM\}\\
T\Gamma_{\phi}\Gamma_{\phi} &= \{(v, \phi_*v, w, \phi_* w)\}&&&T\Gamma_{\phi} &= \{(v,\phi_*v, v, \phi_* v): v \in TM\}\\
T\Delta\Delta &= \{(v,w,w,v): v,w \in TM\}&&&T\Delta &= \{(v,v,v,v): v \in TM\}
\end{align*}
\noindent Therefore, our normal bundles $N(M\times M^{-})$, $N\Delta$, and $N\Gamma_{\phi}$ are as follows.%
\begin{align*}
N(M\times M^{-}) &= \{(v,w,-v,-w): v,w \in TM\}\\
N\Gamma_{\phi} &= \{(v, \phi_* v, -v, -\phi_* v): v \in TM\}\\
N\Delta &= \{(v,-v,-v,v): v \in TM \}
\end{align*}
Observe that these normal bundles are identified with $T(M\times M^{-})$ and its subbundles $T\Gamma_{\phi}$ and $N\Delta$ by projection onto the first two coordinates. So we don't in fact need to work with normal bundles any further; it will suffice to work with $T(M\times M^{-})$ and its Lagrangian subbundles $N\Delta$ over $\Delta$ and $T\Gamma_{\phi}$ over $\Gamma_{\phi}$. By assumption, the vector bundle $TM$ is stably trivializable as a complex vector bundle. This implies that it admits a unitary trivialization with respect to the triple $(\omega, J, g)$, where $\omega$ is the symplectic form on $TM$, $J$ is the almost complex structure, and $g$ is the Hermitian metric induced by $\omega$ and $J$ (cf, e.g., \cite[Section 2.6]{MR1373431}). We choose such a trivialization, as follows.
\begin{align*}
\psi \co TM \oplus \epsilon_{\mathbb C}^{m} &\rightarrow M \times \mathbb C^{k} \\
(x,(v,c)) &\rightarrow (x, \psi_x(v,c))
\end{align*}
Here $\epsilon_{\mathbb C}^m$ is the trivial bundle $M \times {\mathbb C}^m$ over $M$, and $k=n+m$. The expression $(v,c)$ indicates a vector $v \in TM$ and a vector $c \in \epsilon_{\mathbb C}^m$. We will need to keep careful track of both of these vectors to ensure that our final trivialization of $\Upsilon(T(M\times M^{-}))$ is unitary.

Since $TM \oplus \epsilon^m_{\mathbb C}$ is trivializable, there is a map $\Phi \co M \rightarrow Sp(2k)$, where $\Phi(x)$ is the map $\psi \circ (\phi_{*} \oplus \operatorname{I}_m) \circ \psi^{-1}|_{x} \co \mathbb C^{k} \simeq (TM\oplus \epsilon^m_{\mathbb C})_{x} \rightarrow (TM\oplus \epsilon^m_{\mathbb C})_{x} \simeq \mathbb C^{k}$. Here $\operatorname{I}_m$ is the identity map on $\epsilon^m_{\mathbb C}$. Since $U(k)$ is a deformation retract of $Sp(2k)$, $\Phi$ is homotopic to a map to the unitary group.

\begin{proposition}\label{propn:stable}
If $\Phi$ is nullhomotopic, then $(M^4, \Delta\Delta, \Gamma_{\phi}\Gamma_{\phi})$ has a stable normal trivialization.
\end{proposition}

\begin{proof} Recall that we start with a stable unitary trivialization $\psi$ of the tangent bundle $TM$ with respect to $(\omega, J, g)$. This gives us a stable trivialization $\Psi_1$ of $T(M\times M^{-})$, given by
\begin{align*}
\Psi_1 \co TM \oplus TM^{-} \oplus (\epsilon^m_{\mathbb C})^{\oplus 2}  &\rightarrow (M\times M^{-})\times (\mathbb C^k\oplus \mathbb C^k) \\
\left((x,y), (v,w), (c,d)\right) &\mapsto \left(x,y, \psi_x\left(v,\frac{ic+d}{\sqrt{2}}\right), \overline{\psi_y\left(w, \frac{i\overline{c} +\overline{d}}{\sqrt{2}}\right)}\right).
\end{align*}
Note that there is a slight mismatch in the factors above; $(x,y)$ is a point of $M\times M$, whereas $v \in TM_{x}$ and $w \in TM_{y}$, in the spirit of grouping points together and vectors together.

The map $\Psi_1$ is a unitary trivialization with respect to the triple $(\omega \oplus -\omega \oplus \omega_{\mathrm{std}}^{\oplus 2}, \tilde{J}=J\oplus -J \oplus i^{\oplus 2}, g \oplus g \oplus g_{\mathrm{std}}^{\oplus 2})$, on $TM \oplus TM^{-} \oplus (\epsilon^m_{\mathbb C})^{\oplus 2} $. The image of $N\Delta\oplus (i\epsilon^m_{\mathbb R})^{\oplus 2}|_{(x,x)}$ under $\Psi_1$ is the Lagrangian subspace 
\begin{align*}
L &=\left \{\left(\psi_x\left(v,\frac{-r_1+ir_2}{\sqrt{2}}\right), \overline{\psi_x\left(-v,\frac{r_1-ir_2}{\sqrt{2}}\right)}\right): v \in TM, r_1,r_2 \in \mathbb R^m\right \} \subset (\mathbb C^{k}\oplus \mathbb C^k)_{(x,x)} \\
&=\{(d,-\overline{d}): d\in \mathbb C^k\}\subset (\mathbb C^{k}\oplus \mathbb C^k)_{(x,x)}
\end{align*}
In particular, $L$ is constant and does not depend on $x$. Recall that there exists a unitary transformation $A \co \mathbb C^{2k} \rightarrow \mathbb C^{2k}$ such that $A(L) = i\mathbb R^{2k}$, that is, $A$ carries $L$ to the the purely imaginary Lagrangian subspace of $\mathbb C^{2k}$. We compose $A$ with our trivialization $\Psi_1$ to obtain a trivialization $\Psi$ of $TM \oplus TM^{-}\oplus (\epsilon_{\mathbb C}^m)^{\oplus 2}$ which sends $(N\Delta\oplus (i\epsilon_{\mathbb R}^m)^{\oplus 2})|_{(x,x)}$ to $i\mathbb R^{2k} \subset \mathbb C^{2k}$. Finally, we extend the trivialization $\Psi$ to a trivialization $\Psi$ of $\Upsilon(M \times M^{-}) \oplus (\epsilon^m_{\mathbb C})^{\oplus 2}$ via the pullback of the projection map $M\times M^{-} \times I \rightarrow M\times M^{-}$.

Now, recall that our choice of trivialization $\psi$ induces a map $\Phi \co M \rightarrow Sp(2k)$ induced by the action of $\Phi_x = (\psi \circ (\phi_*\oplus \operatorname{I}_m) \circ \psi^{-1})|_{x}$ on $\mathbb C^{2k}$, and we have assumed this map is nulhomotopic. Choose a nulhomotopy $\Phi^t$  between $\Phi^0_x = (\psi \circ (\phi_* \oplus \operatorname{I}_m) \circ \psi^{-1})_{x}$ and $\Phi^1_x := \operatorname{I}_k$. For notational purposes, let us write down the homotopy $\Phi_t$ pulled back to the original vector bundle more explicitly, as follows.
\begin{align*}
\psi^{-1} \circ \Phi_t \circ \psi \co (TM \oplus \epsilon_{\mathbb C}^m) &\rightarrow (TM \oplus \epsilon_{\mathbb C}^m)\\
(x,(v,c)) &\rightarrow (x, (f^t_x(v,c), g^t_x(v,c)))
\end{align*}
Notice that $f^t_x(v,c)$ is a vector in $TM_x$ and $g^t_x(v,c)$ is a complex number. Our goal is to construct Lagrangian subbundles $\Lambda_0$ and $\Lambda_1$ of the restriction of the vector bundle $\Upsilon(M\times M^{-}) \oplus \epsilon^{2m}_{\mathbb C}$ to, respectively, $\Gamma_{\phi} \times I$ and  $\Delta \times I$ such that $(\Lambda_0)|_{\Gamma_{\phi} \times \{0\}} = T\Gamma_{\phi} \oplus (\epsilon^m_{\mathbb R})^{\oplus 2}$ whereas $\Psi((\Lambda_0)|_{\Gamma_{\phi} \times \{1\}}) = \Gamma_{\phi} \times \mathbb R^{2k}$, and $(\Lambda_1)|_{\Delta \times \{0\}} = N\Delta \oplus (i\epsilon^{m}_{\mathbb R})^{\oplus 2}$ and $\Psi((\Lambda_1)|_{\Delta \times \{1\}}) = \Delta \times i\mathbb R^{2k}$. Toward this end, consider the following vector bundles:
\begin{align*}
(\Lambda_1)|_{\Delta \otimes \{t\}} &= N\Delta \oplus (i\epsilon^m_{\mathbb R})^{\oplus 2}\\
(\Lambda_0)|_{\Gamma_{\phi}\times \{t\}} &= \{((x,\phi(x)),(v, f^t_x(v,r_2)), (r_1,g^t_x(v,r_2))): v\in TM, (r_1,r_2) \in (\epsilon_{\mathbb R}^m)^{\oplus 2} \}
\end{align*}
Then we have $(\Lambda_1)|_{\Delta \otimes \{0\}} = N\Delta \oplus (i\epsilon^m_{\mathbb R})^{\oplus 2}$ and $\Psi((\Lambda_1)|_{\{\Delta \otimes \{1\}}) = \Delta \times i\mathbb R^{2k}$, as desired. Moreover, we have $(\Lambda_0)|_{\Gamma_{\phi} \times \{0\}} = \{((v, \phi_*(v)),(r_1, r_2))\} = T\Gamma_{\phi} \otimes (\epsilon_{\mathbb R})^{\otimes 2m}$. It remains to check that $\Psi(\Lambda_0)|_{\Gamma_{\phi} \otimes \{1\}}$ is the correct bundle. Observe that
\begin{align*}
\Psi((\Lambda_0)|_{\{\Gamma_{\phi} \otimes \{1\}})|_{(x, \phi(x))} &= A \cdot \Psi_1((\Lambda_0)|_{\{\Gamma_{\phi} \otimes \{1\}}) \\
												   &= A \cdot \Psi_1((v,v), (r_1,r_2)) \\
												   &= A \cdot iL \\
												   &= \mathbb R^{2k}
\end{align*}
\noindent The last step follows because $A$ is a unitary transformation. So we have described a stably normal trivial structure on $\Upsilon(M \times M^{-})$.\end{proof}

We now summarize the proof of Theorem 1.1.

\begin{proof}[Proof of Theorem 1.1]
By Proposition \ref{propn:stable} , we see that if $\Phi\co M \rightarrow Sp(\infty)$ is nulhomotopic, then $M^4, \Gamma_{\phi}\Gamma_{\phi}$ has a stable normal trivialization. Therefore, by Theorem \ref{thm:localization}, there is a spectral sequence with $E^1$-page $\mathit{HF}(\Gamma_{\phi}\Gamma_{\phi},\Delta\Delta)\otimes \mathbb Z_2((\theta))$ and $E^{\infty}$-page $\mathbb Z_2((\theta))$-isomorphic to $\mathit{HF}(\Gamma_{\phi},\Delta) \otimes \mathbb Z_2((\theta))$. However, by Proposition \ref{propn:isomorphism}, we know there is a natural identification $\mathit{HF}(\Gamma_{\phi}\Gamma_{\phi}, \Delta\Delta) \simeq \mathit{HF}(\Gamma_{\phi^2}, \Delta)$, and the latter theory is $\mathit{HF}(\phi^2)$. Since we also have $\mathit{HF}(\Gamma_{\phi}, \Delta) \simeq \mathit{HF}(\phi)$, we conclude there is a spectral sequence with $E^1$ page identified with $\mathit{HF}(\phi^2)\otimes \mathbb Z_2((\theta))$ and $E^{\infty}$-page isomorphic to $\mathit{HF}(\phi)\otimes \mathbb Z_2((\theta))$.
\end{proof}

\begin{remark} In this paper, we have seen a proof of the identification $\mathit{HF}(\Gamma_{\phi}\Gamma_{\phi}, \Delta\Delta) \simeq \mathit{HF}(\Gamma_{\phi},\Delta)$ which uses only the tools of Lagrangian Floer cohomology. However, this isomorphism closely resembles the simplest case of composition theorems for sequences of Lagrangian correspondences in quilted Floer theories of Wehrheim and Woodward \cite{MR2657646, MR2602853} and Lekili and Lipyanskiy \cite{MR3019714}. In particular, we may regard $\mathit{HF}(\Delta\Delta, \Gamma_{\phi}\Gamma_{\phi})$ in $M^{-}\times M \times M^{-} \times M$ as the Floer cohomology of the cyclic sequence of correspondences $(\Gamma_{\phi}, \Delta, \Gamma_{\phi}, \Delta)$ in, and $\mathit{HF}(\Delta, \Gamma_{\phi^2})$ in $M^{-} \times M$ as the Floer cohomology of the same sequence after two geometric compositions. (There is a convention switch here: in the quilted Floer viewpoint it is usual to let $\mathit{HF}(\phi)$ be the Floer cohomology $\mathit{HF}(\Delta, \Gamma_{\phi})$ in $M^{-}\times M$.) From this point of view, although the isomorphism of Proposition \ref{propn:isomorphism} does not follow directly from any of the existing composition theorems because of the noncompactness of the Lagrangian correspondences involved, it is morally part of the same framework. We speculate that combining these more subtle composition theorems with Seidel--Smith localization theory might produce other interesting spectral sequences from Lagrangian Floer cohomology of the form $HF(L_{01}\circ L_{01}, \Delta) \otimes \mathbb Z_2((\theta))$ to $\mathit{HF}(L_{01},\Delta)\otimes \mathbb Z_2((\theta))$, either by using a compact Lagrangian correspondence $L_{01}$, or by producing composition theorems which are valid for conical Lagrangians.\end{remark}

\acks{I am grateful to Mohammed Abouzaid and Robert Lipshitz for suggesting considering this question and for useful conversations; thanks also to Ko Honda, Ciprian Manolescu, Paul Seidel, and Chris Woodward for their helpful input. Further, I am indebted to the referee for many suggested improvements to the exposition, and for pointing out an error in the original version of the proof of Proposition \ref{propn:stable}.}

\raggedright

\bibliographystyle{amsplain}
\bibliography{bibliography2}

\end{document}